\documentclass[11pt]{amsart}
\usepackage{amsmath} 
\usepackage{amsthm}
\usepackage{amssymb} 
\usepackage{amsfonts}
\usepackage{graphicx} 
\usepackage{inputenc}
\usepackage{upref} 
\usepackage{mathptmx}
\usepackage[T1]{fontenc}
\usepackage{bm}
\usepackage{calc}

\newtheorem{theorem}{Theorem}[section]

\newtheorem{proposition}[theorem]{Proposition}
\newtheorem{corollary}[theorem]{Corollary}

\newtheorem{remark}[theorem]{Remark}

\title[Isoperimetric Problems]{The Sharp Log-Sobolev Inequality on a Compact Interval}  

\author[W. Ghang]{Whan Ghang}
\author[Z. Martin]{Zane Martin}
\author[S. Waruhiu]{Steven Waruhiu}

\date{\today} 

\begin{document}
\maketitle

\begin{abstract}
We provide a proof of the sharp log-Sobolev inequality on a compact interval.
\end{abstract}


\section{Introduction}
\label{intro} 
The Gaussian log-Sobolev inequality, due to A. J. Stam \cite[1959,  Eqn. 2.3]{stam} or Paul Federbush \cite[1969, Eqn. (14)]{federbush}, although often attributed to L. Gross \cite[1975, Cor. 4.2]{gross_gauss}, played a crucial role in Perelman's \cite[2002]{perelman} proof of the Poincar\'{e} Conjecture. We consider log-Sobolev inequalities for finite Lebesgue measure. F. Maggi \cite[2009]{Morgbl} observed that the sharp log-Sobolev inequality on the interval follows from an isoperimetric conjecture of D\'{i}az et al. \cite[2010]{g09}, which remains open, but provided no proof. We found it in Wang \cite[1999]{wang}, who cited Deuschel and Stroock \cite[1990]{deuschel}, who gave a proof of the sharp log-Sobolev inequality on the circle. We then traced this result back to Emery and Yukich \cite[1987, p. 1]{emery&yukich}, Rothaus \cite[1980, Thm. 4.3]{rothaus_1980}, and Weissler \cite[1980, Thm. 1]{weissler}.  Our Theorem 2.2 shows that the interval case follows quickly from the circle case.

\subsection{Acknowledgements}
We thank our advisor Frank Morgan for his patience and invaluable input. We thank our friend Andrew Kelly for substantial contributions to this paper. We thank the National Science Foundation for grants to Professor Morgan and the Williams College "SMALL" Research Experience for Undergraduates, and Williams College for additional funding. We thank the Mathematical Association of America, MIT, the University of Chicago, and Williams for supporting our trip to talk at MathFest 2012.


\setcounter{equation}{0}
\section{Log-Sobolev Inequality on a Compact Interval}

\label{maggi} 

In considering the isoperimetric problem in sectors of the plane with density $ r^p$, D\'{i}az et al. \cite[Cor. 4.24,  Conj. 4.18]{g09} conjectured the following inequality:

\begin{equation}
\label{eq:g09inequality}
\left[ \int_0^1 r^q \;d\alpha \right]^{1/q} \leq \int_0^1 \sqrt{r^2 + (q-1) \frac{r'^2}{\pi^2}} \;d\alpha,
\end{equation}
where $1<q\leq2$. F. Maggi \cite{Morgbl} observed that (\ref{eq:g09inequality}) implies the log-Sobolev inequality of Theorem \ref{logSobolev_Unit_interval} below. Here we observe that Theorem \ref{logSobolev_Unit_interval} follows from the following proposition of Weissler.
\newline

\begin{proposition}
\label{logSobolev_Circle}
\cite[Thm. 1]{weissler} Let $f$ be a non-negative $C^1$ function on the circle $S^1$ of length 1. Suppose $\int_{S^1} f^2 = 1$. Then we have the following sharp inequality:
$$
4\pi^2 \int_{S^1} f^2 \log f \leq \int_{S^1} f'^2.
$$
\end{proposition}

Various proofs are discussed in Section 3.

\setcounter{equation}{0}
\begin{theorem}
\label{logSobolev_Unit_interval}
Let $f$ be a non-negative $C^1$ function on the interval [0,1]. Suppose $\int_0^1 f^2 = 1$. Then we have the following inequality:

\begin{equation}
\label{eq:log_sob_eqn_unit_inter}
\pi^2 \int_0^1 f^2 \log f \leq \int_0^1 f'^2. 
\end{equation}
\end{theorem}

\begin{proof}
Let $f$  be any non-negative $C^1$ function on $[0,1]$ such that $\int_0^1 f^2 = 1$. Define a non-negative piecewise $C^1$ function $g$ on $S^1$ such that 
$$
g(x) =
\begin{cases}
f(2x), & \text{if} \;0\leq x\leq \frac{1}{2} \\
f(2-2x), & \text{if} \;\frac{1}{2}<x\leq 1.
\end{cases}
$$
Then $\int_{S^1}g^2 =1$. By smoothing, Proposition \ref{logSobolev_Circle} applies to $g$.  By simple computation, we have that 
$$
\int_{S^1} g^2 \log g = \int_{0}^{1} f^2 \log f 
$$ 
and 
$$
\int_{S^1} g'^2 =4\int_{0}^{1} f'^2.
$$
The conclusion follows.
\end{proof}

\setcounter{equation}{0}
\begin{remark}
\label{wang_perturbation}
\emph{Feng-Yu Wang \cite[Ex. 1.2]{wang} suggested an alternative proof of 2.2(\ref{eq:log_sob_eqn_unit_inter}), but we don't understand his proof. He considered densities $C_\epsilon \exp(\epsilon \cos \pi x)$ and functions $f_\epsilon =\exp(-\epsilon \cos \pi x)$, with $C_\epsilon$ chosen to make the integral of $f_\epsilon ^2$ equal to 1. Then $f_\epsilon$ satisfies the differential equation
\begin{equation}
\label{differential_eqn_lsf}
f_{\epsilon}''-\pi \epsilon \sin \pi x f'_{\epsilon}=-\pi^2 f_{\epsilon} \log f_{\epsilon}.
\end{equation}
He said that it follows that 2.2(\ref{eq:log_sob_eqn_unit_inter}) holds for those functions and densities with sharp constant $\pi^2$. This might follow if it were known that functions realizing equality exist, but Wang himself \cite[p. 655]{wang} admits that "the author is not sure yet whether there always exists [such a function]." Indeed, in the case of the circle with unit density, there apparently is no such function. Of course, the sharp inequality for density 1 would follow as $\epsilon$ approaches 0.}
\end{remark}

A similar result holds on any interval for a function with root mean square $m$:
\setcounter{equation}{0}

\begin{corollary}
\label{logSobolev_arbitary_interval}
Let $f$ be a non-negative $C^1$ function on the interval [$a,b$]. Suppose \[ \frac{1}{b-a}\int_a^b f^2 = m^2 \]$(m>0).$ Then we have the following inequality:
\begin{equation}
\label{log_sob_eqn_arbitrary_interval}
\frac{\pi^2}{(b-a)^2}\left(\int_a^b f^2 \log f \ -(b-a)m^2 \log m\right) \leq \int_a^b f'^2.
\end{equation}
\end{corollary}

\begin{proof}
Let $f$ be a non-negative $C^1$ function on the interval $[a,b]$ such that 
\[
\frac{1}{b-a}\int_a^b f^2 = m^2>0 
\]  
$(m>0).$ Define a function $g$ on the interval $[0,1]$ as 
$$
g(x)=\frac{1}{m}f((b-a)x+a).
$$
Then $g$ is non-negative and $C^1$. Moreover, we have 
$$
\int_0^1 g(x)^2 \;dx= \int_0^1 \frac{1}{m^2}f((b-a)x+a)^2 dx =\frac{1}{(b-a)m^2} \int_0^a f(y)^2 dy =1.
$$
Therefore, we can apply Theorem \ref{logSobolev_Unit_interval} to the function $g$.  We have
\begin{equation}
\label{log_sob_unit_to_arbitrary}
\frac{\pi^2}{b-a}\int_0^1 g^2 \log g \leq (b-a)\int_0^1 g'^2.
\end{equation}
Note that 
$$
g'(x)=\frac{b-a}{m}f((b-a)x+a).
$$  
By direct calculation, we have
$$
\int_0^1 g'(x)^2dx = \frac{(b-a)^2}{m^2}\int_0^1 f'((b-a)x+a)^2dx =\frac{(b-a)^2}{m^2}\int_a^b f'(x)^2dx.
$$
We also have that 

\begin{align*}
\int_0^1 g(x)^2\log g(x) \;dx &= \frac{1}{m^2}\int_0^1f((b-a)x+a)^2 \log\frac{f((b-a)x+a)}{m}\;dx\\
&=\frac{1}{(b-a)m^2}\int_a^b f(x)^2 \log\frac{f(x)}{m} \;dx\\
&=\frac{1}{(b-a)m^2}\int_a^b f^2 \left(\log f-\log m \right) \\
&=\frac{1}{(b-a)m^2}\left(\int_a^b f^2 \log f - (b-a)m^2 \log m \right).
\end{align*}
Therefore, by plugging these identities into (\ref{log_sob_unit_to_arbitrary}), we have
$$
\frac{\pi^2}{(b-a)m^2}\left(\int_a^b f^2 \log f -(b-a)m^2 \log m \right)\leq\frac{b-a}{m^2}\int_a^b f'^2.
$$
This is equivalent to the desired inequality (\ref{log_sob_eqn_arbitrary_interval}).
\end{proof}

\setcounter{equation}{0}
Corollary \ref{logSobolev_arbitary_interval} can be written in the following form:

\begin{corollary}
Let $f$ be a non-negative $C^1$ function on the interval [$a,b$]. Suppose 
$$
\frac{1}{b-a}\int_a^b f = m>0.
$$
Then we have the following inequality:
$$
\frac{2\pi^2}{(b-a)^2}\left(\int_a^b f \log f \ -m \log m\right) \leq \int_a^b \frac{f'^2}{f}.
$$
\end{corollary}

\begin{proof}
Define a non-negative piecewise $C^1$ function $g$ on the interval [$a,b$] as $g=\sqrt{f}$.
Plugging $g$ into Corollary \ref{logSobolev_arbitary_interval} yields the desired result.
\end{proof}

\begin{proposition} 
In Theorem \ref{logSobolev_Unit_interval}, $\pi^2$ is the best possible constant.
\end{proposition}

\begin{proof}
For any $0< \epsilon <1$, define
$$
f_{\epsilon}(x)=\sqrt{1-\epsilon^2}+\sqrt{2}\epsilon \cos\pi x.
$$
Then by direct computation, we have
\[\lim_{\epsilon \to 0^{+}}\frac{\int_0^1 f'^2_{\epsilon}}{\int_0^1 f_{\epsilon}^2 \log f_{\epsilon}}=\pi^2. \]
Therefore, the constant $\pi^2$ cannot be replaced by a larger constant.
\end{proof}

\begin{remark}
\emph{The function $\cos\pi x$ comes from the equality case of a Wirtinger inequality which follows from the log-Sobolev inequality \cite{Morgbl}.}
\end{remark}

\section{Proofs of the Sharp Log-Sobolev Inequality on the Circle}
We summarize three proofs of Proposition \ref{logSobolev_Circle} given by Rothaus \cite[Thm. 4.3, 1980]{rothaus_1980}, Weissler \cite[Thm. 1, 1980]{weissler}, Emery and Yukich \cite[p. 1, 1987]{emery&yukich}, and Deuschel and Stroock \cite[Rmk. 1.14(i)]{deuschel}.

\subsection{Weissler's Proof}
Weissler proved a stronger result than Proposition \ref{logSobolev_Circle} by using Fourier expansion of functions of period $2\pi$.

\begin{proposition}\cite[1980, Theorem 1]{weissler}
Let $f(\theta)=\sum_{n=-\infty}^{\infty} a_n e^{in\theta}$ be in $L^2$ and suppose $f(\theta)\geq 0$ almost everywhere. Then
$$
\int f^2 \log f \leq \sum_{n=-\infty}^{\infty} |n||a_n|^2 +||f||_2^2 \log ||f||_2
$$
 
\raggedright in the sense that if the right hand side is finite, then so is the left hand side and the inequality holds. ($0^2 \log 0$ is taken to be $0$.)
\end{proposition}

\raggedright Obviously the above inequality is stronger than the following inequality:
$$
\int f^2 \log f \leq \sum_{n=-\infty}^{\infty} |n|^2|a_n|^2 +||f||_2^2 \log ||f||_2
$$

\raggedright which is equivalent to Proposition \ref{logSobolev_Circle} by change of variables as in Corollary \ref{logSobolev_arbitary_interval}.

$\;\;\;\;\;$ Weissler \cite{weissler} cited Rothaus's previous 1978 paper \cite{rothaus_1978} but did not have Rothaus's 1980 paper \cite{rothaus_1980} where Rothaus gave his proof of Proposition \ref{logSobolev_Circle}.

\subsection{Rothaus's Proof}
Rothaus proved Proposition \ref{logSobolev_Circle} by a variational method. He considered an equivalent problem with a positive parameter $\rho$ \cite[Section 4]{rothaus_1980}. If a related constant $b_{\rho}$ is zero then the log-Sobolev inequality on the circle with the constant $2/{\rho}$ holds.  Therefore, Proposition \ref{logSobolev_Circle} is equivalent to showing that $b_{1/2\pi^2}$ is zero. 

$\;\;\;\;\;$ For each $b_{\rho}$, he showed that a minimizing function exists, is positive and satisfies a related differential equation \cite[Thm. 4.2]{rothaus_1980}.  Moreover, for $\rho>1/2\pi^2$, the only positive solution to the differential equation is the constant function 1 \cite[Thm. 4.3]{rothaus_1980} and hence $b_{\rho}$ is zero. Therefore in the limit  $b_{1/2\pi^2}$ is zero, and Proposition \ref{logSobolev_Circle} follows.

$\;\;\;\;\;$ Rothaus \cite{rothaus_1980} cited Weissler's paper \cite{weissler} and said that "A result related to Theorem 6.3 appears in \cite{weissler}."

\subsection{Emery and Yukich's Proof}
Emery and Yukich \cite[1987, p. 1]{emery&yukich} proved Proposition \ref{logSobolev_Circle} by using estimates deploying the Brownian motion semi-group.

$\;\;\;\;\;$ Emery and Yukich \cite{emery&yukich} cited both Weissler \cite{weissler} and Rothaus \cite{rothaus_1980}.

\subsection{Deuschel and Stroock's Proof}
Deuschel and Stroock considered the log-Sobolev inequality in general spaces with densities.  As a special case they  proved \cite[Rmk. 1.14(i)]{deuschel} that the log-Sobolev constant for the circle of length 1 with Lebesgue measure is the first eigenvalue of the Laplacian, namely $4\pi^2$ $\left(\text{corresponding to the first eigenfunction} \sin 2\pi x\right)$. 

$\;\;\;\;\;$ Deuschel and Stroock \cite{deuschel} cited Emery and Yukich \cite{emery&yukich}.

\bibliographystyle{abbrv}
\bibliography{main}

\bigskip
Whan Ghang
\newline 
Department of Mathematics
\newline 
Massachusetts Institute of Technology
\newline
ghangh@MIT.edu
\newline
\newline
Zane Martin
\newline 
Department of Mathematics and Statistics 
\newline 
Williams College 
\newline
zkm1@williams.edu
\newline
\newline
Steven Waruhiu
\newline 
Department of Mathematics
\newline 
University of Chicago
\newline
waruhius@uchicago.edu

\end{document}